\documentclass[a4paper,12pt,leqno,twoside]{article}
\usepackage{amssymb,amsbsy,amsmath,amsfonts,amssymb,amscd,
	graphics,color,footmisc,fancyhdr,multicol,fancybox,
	graphicx,mathrsfs,rotating,ifthen,wasysym}

\usepackage{times}

\usepackage[dvipsnames,svgnames,x11names]{xcolor}
\usepackage[toc,page]{appendix}
\usepackage{pdfpages}
\usepackage{multirow}
\usepackage{nccrules}
\usepackage{textcomp}
\usepackage{comment}
\usepackage{framed}
\usepackage[all]{xy}
\usepackage{graphicx}
\usepackage{pgf,tikz}
\usepackage{mathrsfs}
\usetikzlibrary{arrows}
\usepackage{lipsum}
\usepackage{mathtools}

\definecolor{black}{cmyk}{1.,1.,1.,1.0}
\definecolor{blue}{cmyk}{1.,1.,0.,0.63}
\definecolor{green}{cmyk}{1.,0.,1.,0.63}

\usepackage{hyperref}
\hypersetup{
       unicode=true,          
     pdftoolbar=true,        
    pdfmenubar=true,        
    pdffitwindow=false,     
    pdfstartview={FitH},    
       pdfproducer={}, 
    pdfcreator={pdfLaTeX},   
    pdfproducer={}, 
    pdfnewwindow=true,      
    colorlinks=false,       
    linkcolor=blue,          
    citecolor=green,        
    filecolor=magenta,      
    urlcolor=cyan           
}
\usepackage[top=2.5cm, bottom=2.5cm, left=2.5cm, right=2.5cm]{geometry}
\usepackage{amssymb,amsmath,amsthm}

\DeclareMathOperator{\dif}{d}
\DeclareMathOperator{\rank}{rank}

\DeclareMathOperator{\supp}{supp}
\DeclareMathOperator{\ord}{ord}

\newcommand{\vf}{\vfill\end{document}}
\newcommand{\explain}[1]{\text{\scriptsize\sf [#1]}}
\theoremstyle{definition}
\newtheorem{defi}{Definition}[section]
\theoremstyle{plain}
\newtheorem{thm}{Theorem}[section]
\newtheorem{lem}{Lemma}[section]
\newtheorem{pro}{Proposition}[section]
\newtheorem{cor}{Corollary}[section]

\newtheorem{rem}{Remark}[section]
\theoremstyle{plain}
\newcommand{\thistheoremname}{}
\newtheorem*{genericthm*}{\thistheoremname}
\newenvironment{namedthm*}[1]{\renewcommand{\thistheoremname}{#1}%
\begin{genericthm*}}
{\end{genericthm*}}

\newtheoremstyle{named}{}{}{\itshape}{}{\bfseries}{.}{.5em}{\thmnote{#3's }#1}
\theoremstyle{named}

\title{\bf On the Weyl-Ahlfors theory of derived curves}
\providecommand{\keywords}[1]{\textbf{\textit{Keywords:}} #1}
\providecommand{\subject}[1]{\textbf{\textit{Mathematics Subject Classification 2010:}} #1}

\author{Dinh Tuan Huynh, Song-Yan Xie\footnote{\ \ partially supported by  NSFC Grant No.~$11688101$}}

\newcommand{\Addresses}{{
		
		\footnotesize
		\noindent
		\textsc{Dinh Tuan Huynh, Hua Loo-Keng center for Mathematical Sciences, Academy of Mathematics and System Science, Chinese Academy of Sciences, Beijing 100190, China \& Department of Mathematics, University of Education, Hue University, 34 Le Loi St., Hue City, Vietnam}
		\par\nopagebreak
		\noindent
		\textit{E-mail address}: \texttt{dinh-tuan.huynh@amss.ac.cn}
		\newline
		
		\noindent
		\textsc{Song-Yan Xie, Academy of Mathematics and System Science \& Hua Loo-Keng Key Laboratory
			of Mathematics, Chinese Academy of Sciences, Beijing 100190, China}
		\noindent
		\par\nopagebreak
		\noindent
		\textit{E-mail address}: \texttt{xiesongyan@amss.ac.cn}
	}}
\begin{document}
	 \date{}
\maketitle
\begin{abstract}
	For derived curves intersecting a family of decomposable hyperplanes in {\sl subgeneral position}, we obtain an analog of the Cartan-Nochka Second Main Theorem, generalizing a classical result of Fujimoto
	about decomposable hyperplanes  in {\sl general position}.
\end{abstract}
\keywords  {Value distribution theory}, {Second Main Theorem}, {entire curves}, {Nochka weights}, {defect relation}, {subgeneral position}, {Wronskians}

\noindent
\subject{32H30}, {32A22}
\section{Introduction}
Value distribution theory was started by Nevanlinna~\cite{Nevanlinna1925} by relating the intersection frequency of a holomorphic map $f:\mathbb{C}\rightarrow\mathbb{P}^1(\mathbb{C})$ with  $q\geq 3$ distinct points in $\mathbb{P}^1(\mathbb{C})$, and the growth rate of $f$. This quantifies the classical little Picard theorem, and also generalizes the fundamental theorem of algebra from polynomials to meromorphic functions.

In higher dimension, Cartan~\cite{Cartan1933} explored Nevanlinna theory in the setting of a linearly nondegenerate entire curve
$f:\mathbb{C}\rightarrow\mathbb{P}^n(\mathbb{C})$
together with a family of $q\geq n+2$ hyperplanes $\{H_i\}_{i\,=\,1,\dots,\, q}$ in {\sl general position}, and he obtained a second main theorem:
\begin{equation}
\label{Cartan SMT}
(q-n-1)
\,
T_f(r)
\,
\leq
\,
\sum_{i=1}^q N_f^{[n]}(r,H_i)+S_f(r),
\end{equation}
 (see section 2 for meanings of these notations) by introducing a Wronskian technique, which is indispensable in the subject~\cite{Ru2009,Ru2018}. For hyperplanes $\{H_i\}_{i\,=\,1,\dots,\, q}$ in $N$-subgeneral position, i.e., there exists some embedding $\mathbb{P}^n(\mathbb{C})\hookrightarrow \mathbb{P}^N(\mathbb{C})$ such that  $\{H_i=H_i'\cap \mathbb{P}^n(\mathbb{C})\}_{i\,=\,1,\dots, \,q}$ are the restrictions of hyperplanes $\{H_i'\subset \mathbb{P}^N(\mathbb{C})\}_{i\,=\,1,\dots,\, q}$ in general position, Cartan anticipated that there shall be
 \begin{equation}
 \label{Nochka SMT}
 (q-2N+n-1)
 \,
 T_f(r)
 \,
 \leq
 \,
 \sum_{i=1}^q N_f^{[n]}(r,H_i)+S_f(r),
 \end{equation}
  and this conjecture was proved by Nochka~\cite{Nochka1983}
   by means of the  so-called Nochka weights.

Meanwhile, independently, Weyl's~\cite{Weyl1938,Weyl1943}  restarted the study of value distribution of entire curves in $\mathbb{P}^N(\mathbb{C})$
with respect to high codimension projective subspaces, by introducing the associated derived curves which assign every point $f(z)$ with the osculating $k^{\text{th}}$-planes passing through that point (see Subsection 2.3).
In the same vein,
Ahlfors~\cite{Ahlfors1941} successfully established a second main theorem type estimate for derived curves, which embraces the inequality~\eqref{Cartan SMT} of Cartan when $k=0$ and the targets are hyperplanes.  The reader is referred to ~\cite{Wu1970, Shabat1985} for  expositions about Weyl--Ahlfors' theory.

Since then Weyl-Ahlfors theory has  much progress. Notably, Stoll~\cite{Stoll53,Stoll54} studied meromorphic maps from parabolic spaces to projective spaces; Cowen--Griffiths \cite{Griffiths-Cowen76} gave a simplified proof of Ahlfors' result using negative curvature; Fujimoto~\cite{Fujimoto1982,Fujimoto1993} established a second main theorem for derived curves of linearly nondegenerate entire curves with optimal truncation level;
Chen~\cite{Chen1990} generalized the Ahlfors' result for degenerated  entire curves.

Inspiring by the works~\cite{Fujimoto1982, Chen1990}, it would be natural to seek a second main theorem for derived curves, having optimal truncation level,  without assuming the nondegeneracy of the  entire curves. Here is our result in this direction, which is a generalization of Cartan-Nochka's Second Main Theorem.

\begin{thm}
	\label{Generalization of Cartan-Nochka's Second Main Theorem}
	Let $f:\mathbb{C}\rightarrow\mathbb{P}^N(\mathbb{C})$ be an entire holomorphic curve, and let $\mathbb{P}^n(\mathbb{C})\subset\mathbb{P}^N(\mathbb{C})$ be the smallest linear projective subspace  containing  $f(\mathbb{C})$. For a fixed integer $k=0,1,\dots, n$, let $A_1,\dots, A_q$ be  $q$ decomposable hyperplanes of $\mathbb{P}
	\big(
	\Lambda^{k+1}(\mathbb{C}^{N+1})
	\big)$ in general position such that none of them contains the induced Pl\"ucker subset  $\mathbb{P}
	\big(
	\Lambda^{k+1}(\mathbb{C}^{n+1})
	\big)\subset
	\mathbb{P}
	\big(
	\Lambda^{k+1}(\mathbb{C}^{N+1})
	\big)$. Then the $k$-th derived curve $F_k$ of $f$ satisfies
	\begin{align*}
	\Bigg(
	q
	-
	2
	\begin{pmatrix}
	N+1\\k+1
	\end{pmatrix}
	+
	\begin{pmatrix}
	n+1\\k+1
	\end{pmatrix}
	\Bigg)
	T_{F_k}(r)\leq\sum_{i=1}^q N_{F_k}^{[(k+1)(n-k)]}(r,A_i)
	+
	S_{F_k}(r).
	\end{align*}
\end{thm}

In fact, this result follows directly from the following stronger statement (see Remark~\ref{explain why main theorem implies thm 1}).

\begin{namedthm*}{Main Theorem}
	Let $f:\mathbb{C}\rightarrow\mathbb{P}^n(\mathbb{C})$ be a linearly nondegenerate  entire holomorphic curve. For a fixed integer $k=0,1,\dots, n$, let $A_1,\dots, A_q\subset \mathbb{P}
	\big(
	\Lambda^{k+1}(\mathbb{C}^{n+1})
	\big)$ be  $q$ decomposable hyperplanes
	such that any $\mathfrak{N}$ of them have empty intersection. Then the $k$-th derived curve $F_k$ of $f$ satisfies
	\begin{align}
		\label{smt statement}
	(
		q
		-
		2
		\mathfrak{N}
		+
		\mathfrak{n}
		)\,
		T_{F_k}(r)\leq\sum_{i=1}^q N_{F_k}^{[(k+1)(n-k)]}(r,A_i)
		+
		S_{F_k}(r),
	\end{align}
	where $\mathfrak{n}:=\dim_{\mathbb{C}} 
	\Lambda^{k+1}(\mathbb{C}^{n+1})=\left( \substack{
		n+1\\k+1
	}\right)$.

\end{namedthm*}

Terminologies and notation will be explained in Section~\ref{section: preparation}, while a complete proof will be reached in Section~\ref{proof of the Main Theorem}, which depends on  classical techniques of Cartan's Wronskian~\cite{Cartan1933}, Nochka's weight~\cite{Nochka1983} and Fujimoto's vanishing order estimates~\cite{Fujimoto1982}. 
Whence a defect relation~\eqref{defect relation} can be concluded in Section~\ref{section: applications}. Theorem~\ref{Generalization of Cartan-Nochka's Second Main Theorem} improves a previous result of Chen~\cite{Chen1990} by providing an effective truncation level $(k+1)(n-k)$, which is optimal as shown by an  example of Fujimoto~\cite{Fujimoto1982}.
For $k=0$, we recover the celebrated Cartan-Nochka's Theorem \cite{Nochka1983}. When $\mathfrak{n}=\mathfrak{N}$, our defect relation~\eqref{defect relation} coincides with a result of Fujimoto~\cite{Fujimoto1982} for decomposable hyperplanes in general position. 


\section*{Acknowledgments}
Both authors are grateful to the Academy of Mathematics and System Sciences in Beijing for nice working conditions. Huynh also wants to thank the support from Hue University.

\section{Preliminaries}
\label{section: preparation}

\subsection{Nevanlinna theory}

We denote by $\Delta_r\subset \mathbb{C}$ the disk  of radius $r>0$ centered at the origin.
Fix a truncation level $m\in \mathbb{N}\cup \{\infty\}$, 
for an effective divisor $E=\sum_i\alpha_i\,a_i$ on $\mathbb{C}$ where $\alpha_i\geq 0$, $a_i\in\mathbb{C}$,  the $m$-truncated degree of the divisor $E$ on  the disk $\Delta_r$ is given by
\[
n^{[m]}(r,E)
:=
\sum_{a_i\in\Delta_r}
\min
\,
\{m,\alpha_i\},
\]
the \textsl{truncated counting function at level} $m$ of $E$ is then defined by taking the logarithmic average
\[
N^{[m]}(r,E)
\,
:=
\,
\int_1^r \frac{n^{[m]}(t, E)}{t}\,\dif\! t
\eqno
{{\scriptstyle (r\,>\,1)}.}
\]
When $m=\infty$, for abbreviation we  write $n(t,E)$, $N(r,E)$ for $n^{[\infty]}(t,E)$, $N^{[\infty]}(r,E)$.

Let $f\colon\mathbb{C}\rightarrow \mathbb{P}^n(\mathbb{C})$ be an entire holomorphic curve having a reduced representation $f=[f_0:\cdots:f_n]$ in the homogeneous coordinates $[z_0:\cdots:z_n]$ of $\mathbb{P}^n(\mathbb{C})$. Let $D=\{Q=0\}$ be a divisor in $\mathbb{P}^n(\mathbb{C})$ defined by a homogeneous polynomial $Q\in\mathbb{C}[z_0,\dots,z_n]$ of degree $d\geq 1$. If $f(\mathbb{C})\not\subset D$, then $f^*D=\sum_{a\in\mathbb{C}}\ord_af^*Q$ is a divisor on $\mathbb{C}$. We then define the \textsl{truncated counting function} of $f$ with respect to $D$ as
\[
N_f^{[m]}(r,D)
\,
:=
\,
N^{[m]}\big(r,f^*D\big),
\]
which measures the intersection frequency of $f(\mathbb{C})$ with $D$. If $f^*D=\sum_i\mu_ia_i$, where $\mu_i>0
$ and $\mu=\min\{\mu_i\}$, then we say that  $f$ is {\sl completely $\mu$--ramified} over $D$, with the convention that $\mu=\infty$ if $f(\mathbb{C})\cap\supp D=\emptyset$. Next,
the \textsl{proximity function} of $f$ associated to the divisor $D$ is given by
\[
m_f(r,D)
\,
:=
\,
\int_0^{2\pi}
\log
\frac{\big\Vert f(re^{i\theta})\big\Vert_{\max}^d\,
	\Vert Q\Vert_{\max}}{\big|Q(f)(re^{i\theta})\big|}
\,
\frac{\dif\!\theta}{2\pi},
\]
where $\Vert Q\Vert_{\max}$ is the maximum  absolute value of the coefficients of $Q$ and where
\begin{equation}
\label{| |max definition}
\big\Vert f(z)\big\Vert_{\max}
:=
\max
\{|f_0(z)|,\dots,|f_n(z)|\}.
\end{equation}
Since $\big|Q(f)\big|\leq
\left(\substack{d+n\\ n}
\right)\,
 \Vert Q\Vert_{\max}\cdot\Vert f\Vert_{\max}^d$, we see that $m_f(r,D)\geq O(1)$ is bounded  from below by some constant.
Lastly, the \textsl{Cartan order function} of $f$ is defined by
\begin{align*}
T_f(r)
:=
\frac{1}{2\pi}\int_0^{2\pi}
\log
\big\Vert f(re^{i\theta})\big\Vert_{\max} \dif\!\theta.
\end{align*}

The Nevanlinna theory is then established by comparing the above three functions. It consists of two fundamental theorems (for a comprehensive exposition, see Noguchi-Winkelmann \cite{Noguchi-Winkelmann2014}).

\begin{namedthm*}{First Main Theorem}\label{fmt} Let $f\colon\mathbb{C}\rightarrow \mathbb{P}^n(\mathbb{C})$ be a holomorphic curve and let $D$ be a hypersurface of degree $d$ in $\mathbb{P}^n(\mathbb{C})$ such that $f(\mathbb{C})\not\subset D$. Then one has the estimate
	\[
	m_f(r,D)
	+
	N_f(r,D)
	\,
	=
	\,
	d\,T_f(r)
	+
	O(1)
	\]
	for every $r>1$,
	whence
	\begin{equation}
	\label{-fmt-inequality}
	N_f(r,D)
	\,
	\leq
	\,
	d\,T_f(r)+O(1).
	\end{equation}
\end{namedthm*}

Hence the counting function is bounded from above by some multiple of the order function. The reverse direction is usually much harder, and one often needs to take the sum of the counting functions of many divisors. Such types of estimates are so-called {\sl second main theorems}.

Throughout this paper, for an entire curve $f,$ the notation $S_f(r)$ means a real function of $r \in \mathbb{R}^+$ such that 
\[
S_f(r) \leq
O(\log(T_f(r)))+ \epsilon \log r
\]
for every positive constant $\epsilon$ and every $r$ outside of a subset (depending on $\epsilon$) of finite Lebesgue measure of $\mathbb{R}^+$. In the case where $f$ is rational, we understand that $S_f(r)=O(1)$. In any case we always have

\[
\liminf_{r\rightarrow\infty}\dfrac{S_f(r)}{T_f(r)}
=
0.
\]

\subsection{Grassmann algebra}
Let $E$ be a $\mathbb{C}$-vector space of dimension $M+1$.
The graded exterior algebra $\Lambda^{\bullet} E=\oplus_{k=0}^{M}\,\Lambda^{k} E$,
equipped with the exterior wedge product, is called a Grassmann algebra.  Every element in $\Lambda^{k}E$ is called a $k$-vector, and it is said to be {\sl decomposable} if it can be written neatly as  $a_1\wedge\dots\wedge a_k$ for some $k$ vectors $a_1,\dots,a_k\in E$. 

Given a basis $\{e_0,\dots,e_M\}$ of $E$, then  $\Lambda^{k+1} E$ has the basis
$\{e_{i_0}\wedge\dots\wedge e_{i_k}\}_{0\,\leq\, i_0\,<\,i_1\,<\,\cdots\,<\,i_k\,\leq\, M}$.
In this coordinate system, for $k+1$ vectors
  $a_i=\sum_{j=0}^{M}a_{i,j} e_j$ where $i=0,\dots, k$, direct  computation shows:
\begin{equation}
\label{wedge product formula}
a_0\wedge\dots\wedge a_k
=
\sum_{0\,\leq\, i_0\,<\,i_1\,<\,\cdots\,<\,i_k\,\leq\, M} a({i_0},\dots,{i_k})\,e_{i_0}\wedge\dots\wedge e_{i_k},
\end{equation}
where 
$
a({i_0},\dots,{i_k})
:=
\det
\big(
(a_{\alpha,i_{\beta}})_{0\,\leq\,\alpha,\,\beta\,\leq\, k}
\big)
$.

\subsection{Derived curves}
\label{Derived curves}
Let $f:\mathbb{C}\rightarrow\mathbb{P}^n(\mathbb{C})$ be a linearly nondegenerate entire holomorphic curve with a reduced representation $f=[f_0:\dots:f_n]$ in the homogeneous coordinates. Note that
  \begin{equation}
  \label{lifting f}
  \widetilde{f}:=(f_0,\dots,f_n):\mathbb{C}\rightarrow\mathbb{C}^{n+1}\setminus\{0\}
  \end{equation}
   provides a lifting of $f$ along the natural projection  $\pi:\mathbb{C}^{n+1}\setminus\{0\}\rightarrow\mathbb{P}^n(\mathbb{C})$.  For $k=1,\dots,n$, to construct a $k^{\text{th}}$-derived curve  we first collect all the derivatives 
\begin{equation}
\label{lift of f, derivatives}
\widetilde{f}^{(\ell)}=(f_0^{(\ell)},\dots,f_n^{(\ell)})
:\ \ \mathbb{C} \ \ \longrightarrow
\ \ \mathbb{C}^{n+1}
\ \ \ \ \ \ \ \ \ \ \ \ \ \ \ 
{\scriptstyle(0\,\leq\, \ell\leq\, k),}
\end{equation}
 up to order $k$,
and then take their wedge product
\[
\widetilde{F}_k:=\widetilde{f}^{(0)}\wedge\cdots\wedge\widetilde{f}^{(k)}
\  :
\ \ \mathbb{C}\ \
\longrightarrow\ \ 
\Lambda^{k+1}(\mathbb{C}^{n+1}).
\]
Relating to the standard  basis
$\{e_i\}_{i\,=\,0,\dots,\, n}$ of $\mathbb{C}^{n+1}$,  by~\eqref{wedge product formula} there holds
\begin{equation}
\label{Plucker coordinates of F_k}
\widetilde{F}_k
=
\sum_{0\,\leq\, i_0\,<\,i_1\,<\,\dots\,<\,i_k\,\leq\, n} W(f_{i_0},\cdots,f_{i_k})\,e_{i_0}\wedge\dots\wedge e_{i_k}
\end{equation}
in Pl\"ucker coordinates,
where 
$
W(f_{i_0},\dots,f_{i_k})
:=
\det
\big(
(f_{i_{\beta}}^{(\alpha)})_{\alpha,\,\beta\,=\,0,\dots,\, k }
\big)
$
is a standard Wronskian. For the purpose of descending the image of $\widetilde{F}_k$ 
along the natural projection
\[
\pi: \Lambda^{k+1}(\mathbb{C}^{n+1})\setminus \{0\}
\rightarrow
\mathbb{P}
\Big(
\Lambda^{k+1}(\mathbb{C}^{n+1})
\Big),
\] 
we now cancel the common zeros of all the obtained Wronskians by an auxiliary holomorphic function $g$ satisfying
\[
(g)_0=\min_{0\,\leq\, i_0\,<\,i_1\,<\,\cdots\,<\,i_k\,\leq\, n}
\Big(
W(f_{i_0},\dots,f_{i_k})
\Big)_0.
\]
Hence the quotient succeeds
\[
\overline{F}_k:=\widetilde{F}_k/g \ \ :\ \ 
\mathbb{C}\ \
\longrightarrow\ \ 
\Lambda^{k+1}(\mathbb{C}^{n+1})\setminus \{0\}.
\] 

\begin{defi}
	The $k^{\text{th}}$-{\sl derived curve} of $f$ is
	\[
	F_k:=\pi\circ\overline{F}_k:\mathbb{C}\longrightarrow
	\mathbb{P}
	\Big(
	\Lambda^{k+1}(\mathbb{C}^{n+1})
	\Big).
	\]
\end{defi}

Recall that the {\sl Cartan's order function} of $F_k$ is given by
\begin{align*}
T_{F_k}(r)
&=
\dfrac{1}{2\pi}
\int_0^{2\pi}
\log\|F_k(re^{i\theta})\|_{\max}
\dif\theta\\
&=
\dfrac{1}{2\pi}
\int_0^{2\pi}
\log
\max_{0\,\leq\, i_0\,<\,i_1\,<\,\cdots\,<\,i_k\,\leq\, n}
\bigg|\dfrac{W(f_{i_0},\dots,f_{i_k})}{g}(re^{i\theta})\bigg|\dif\theta.
\end{align*}

It is known that all the derived curves have the same growth rate (see e.g. \cite{Fujimoto1993}):
\begin{equation}
\label{same growth rate for derived curves}
T_{F_k}=
O(T_{F_{\ell}})
\ \ \ \ \ \ \ \ \ \ 
{\scriptstyle(\,0\,\leq\, k,\,\ell\,\leq\, n\,)}.
\end{equation}

A {\sl decomposable hyperplane}
\[
A:=
\pi\,
\{Z\in\Lambda^{k+1}(\mathbb{C}^{n+1})\
: \ 
A^*(Z)=0\}
\subset \mathbb{P}
\big(
\Lambda^{k+1}(\mathbb{C}^{n+1})
\big)
\]
is the dual of
 a nonzero decomposable $(k+1)$-vector 
 \[
 A^*=a_0\wedge\dots\wedge a_k
 \ \ 
 \in\ \ \Lambda^{k+1}(\mathbb{C}^{n+1})^{\vee}
 \cong (\Lambda^{k+1} \mathbb{C}^{n+1})^{\vee}
 \eqno
 {\scriptstyle (\,a_0,\,\dots,\,a_k\,\in\,(\mathbb{C}^{n+1})^{\vee}\,)}.
 \]

 We claim that the image of the derived curve $F_k(\mathbb{C})$ is not contained in any decomposable hyperplane. Indeed, for any $A^*$ above, writing each $a_i$ in the standard dual basis $\{e_j^*\}_{j\,=\,0,\dots,\, n}$ of $(\mathbb{C}^{n+1})^{\vee}$ 
 as
 $a_i=\sum_{j=0}^n a_{i,j}e_j^*$,  by formula~\eqref{wedge product formula} we have
 \begin{align}
 \label{use Cauchy-Binet formula}
 A^*(\widetilde{F}_k)
 &
 =
 \sum_{0\,\leq\, i_0\,<\,i_1\,<\,\dots\,<\,i_k\,\leq\, n}
 \det
 \big(
 (a_{\alpha,\,i_{\beta}})_{0\,\leq\,\alpha,\,\beta\,\leq\, k}
 \big)\,
 e_{i_0}^*\wedge\dots\wedge e_{i_k}^*(\widetilde{F}_k)
 \\
 \text{[recall~\eqref{Plucker coordinates of F_k}]}
 \ \ \ \ \ \ \
 &
 =
 \sum_{0\,\leq\, i_0\,<\,i_1\,<\,\dots\,<\,i_k\,\leq\, n} 
 \det
 \big(
 (a_{\alpha,\,i_{\beta}})_{0\,\leq\,\alpha,\,\beta\,\leq\, k}
 \big)
 \,
 \det
 \big(
 (f_{i_{\beta}}^{(\alpha)})_{0\,\leq\,\alpha,\,\beta\,\leq\, k}
 \big)
 \notag
 \\
 \text{[by Cauchy-Binet Formula]}
 \ \ \ \ \ \ \
 &
 =
 \det
 \big(
 (h_{{\beta}}^{(\alpha)})_{0\,\leq\,\alpha,\,\beta\,\leq\, k}
 \big),
 \notag
 \end{align}
 where each $h_i=\sum_{j=0}^n a_{i,\,j} f_j=a_i\circ \widetilde{f}$ for $i=0,\dots ,k$.
 The linearly independence of  $f_0,\dots, f_n$  as well as that of $a_0,\dots, a_{k}$
 guarantee that $h_0,\dots, h_k$ are also linearly independent, whence the Wronskian 
 \[\det
 \big(
 (h_{{\beta}}^{(\alpha)})_{0\,\leq\,\alpha,\,\beta\,\leq\, k}
 \big)\not\equiv 0,
 \]
 i.e., 
 $F_k(\mathbb{C})$ is not contained in the decomposable hyperplane 
 $A$ defined by $A^*$.

 Therefore, we define the  $m$--truncated {\sl counting function} of $F_k$ with respect to  $A$ as
\[
N_{F_k}^{[m]}(r,A):=N^{[m]}(r,(A^*\circ \overline{F}_k)_0).
\]
The {\sl $m$--defect} of $F_k$ with respect to $A$ is then defined by 
\[
\delta^{[m]}_{F_k}(A):=\liminf_{r\rightarrow\infty}
\bigg(1-\dfrac{N_{F_k}^{[m]}(r,A)}{T_{F_k}(r)}\bigg),
\]
which according to the  First Main Theorem satisfies
$
0\leq \delta^{[m]}_{F_k}(A)\leq 1
$.

\subsection{ Nochka's weights}
 Let $N\geq n$ be two positive integers.
Let $H_1,\dots, H_q
\subset \mathbb{P}^n(\mathbb{C})$ be   $q\geq N+1$ hyperplanes defined by the linear
forms $h_1^*,\dots,h_q^*\in
(\mathbb{C}^{n+1})^{\vee}$, respectively.

\smallskip\noindent
{\bf Conventions.}
Denote by $Q$ the index set $\{1,\dots,q\}$. For a subset $R\subset Q$, denote by $|R|$ its cardinality and by $\rank(R)$ the dimension of the linear subspace of
$(\mathbb{C}^{n+1})^{\vee}$ generated by $\{h_i\}_{i\in R}$.

\begin{defi}
	The family $\{H_i\}_{i\,=\,1,\dots,\, q}$ is said to be in $N$-{\sl subgeneral position} if any $N+1$ hyperplanes in this family have {\it empty} intersection. When $N=n$, this family is said to be in {\sl general position}.
\end{defi}

\begin{rem}
	\label{explain why main theorem implies thm 1}
Keeping the assumptions as in the statement of Theorem~\ref{Generalization of Cartan-Nochka's Second Main Theorem}, we may regard $f$ as a linearly non-degenerate curve $f:\mathbb{C}\rightarrow\mathbb{P}^n(\mathbb{C})\hookrightarrow\mathbb{P}^N(\mathbb{C})$, which induces the derived curve 
	\[
	F_k:\mathbb{C}\rightarrow\mathbb{P}^{\mathfrak{n}-1}(\mathbb{C})\hookrightarrow\mathbb{P}^{\mathfrak{N}-1}(\mathbb{C}),\]
	 where $	\mathfrak{N}=\begin{pmatrix}
		N+1\\k+1
	\end{pmatrix}
	$ and $
	\mathfrak{n}=
	\begin{pmatrix}
		n+1\\k+1
	\end{pmatrix}$. Still using $A_i$ to denote the cut loci $A_i\cap \mathbb{P}^{\mathfrak{n}-1}(\mathbb{C})$, then $\{A_i\}_{i\in Q}$ is a family of $q$ hyperplanes in $\big(\mathfrak{N}-1\big)$--subgeneral position of $\mathbb{P}^{\mathfrak{n}-1}(\mathbb{C})$. Hence Theorem~\ref{Generalization of Cartan-Nochka's Second Main Theorem} is a direct consequence of the Main Theorem.
\end{rem}

Here is the main tool in Nochka's  resolution~\cite{Nochka1983} of Cartan's conjecture.

\begin{thm}
	\label{Nochka weights and constant}
	Let $\{H_i\}_{i\,=\,1,\dots,\, q}$ be a family of $q\geq 2N-n+1$ hyperplanes in $N$-subgeneral position of $\mathbb{P}^n(\mathbb{C})$. Then there exists a family of rational constants $\{\omega(i)\}_{i\,=\,1,\dots,\, q}$ satisfying the following conditions:
	\begin{itemize}
		\item[i)] $0\leq\omega(i)\leq 1$ for all $i=1,\dots, q$;
		\item[ii)] set $\widetilde{\omega}:=\max_{1\,\leq\, i\,\leq\, q}\omega(i)$, then
		\[
		\sum_{i=1}^q\omega(i)
		=
		\widetilde{\omega}(q-2N+n-1)+n+1;
		\]
		\item[iii)]  $\frac{n+1}{2N-n+1}
		\leq \widetilde{\omega}\,\leq\frac{n}{N};$
		\item[iv)] if $R$ is a subset of $Q$ with $0< |R|\leq N+1$, then 
		\begin{equation}
		\label{Nochka rank inequatlity}
		\sum_{i\in R}\omega(i)\leq\rank(R).
		\end{equation}
	\end{itemize}
\end{thm}

The constants $\omega(j)$ are called {\sl Nochka's weights} and $\widetilde{\omega}$ is called {\sl Nochka's constant} of the family $\{H_i\}_{i\,=\,1,\dots,\, q}$. They satisfy the following key property (c.f. \cite[Lem. 4.1.17]{Noguchi-Winkelmann2014}).

\begin{pro}
	\label{Nochka estimate}
		Let $\{H_i\}_{1\leq i\leq q}$ be a family of $q\geq 2N-n+1$ hyperplanes in $N$-subgeneral position of $\mathbb{P}^n(\mathbb{C})$. Let $a_1,\dots, a_q\geq 1$ be arbitrary constants. If $R$ is a subset of $Q$ having cardinality \[
		0<|R|\leq N+1,
		\]
	 then there exist distinct indices $i_1,\dots,i_{\rank(R)}\in R$ such that 
	 \[
	 \rank(\{i_1,\dots,i_{\rank(R)}\})=\rank(R)
	  \quad\text{and}
	  \quad
	\prod_{i\in R}\alpha_i^{\omega(i)}
		\leq
		\prod_{k=1}^{\rank(R)}\alpha_{i_k}.
		\]
\end{pro}

\section{Proof of the Main Theorem}
\label{proof of the Main Theorem}

\subsection{Notation and conventions}
Fix a reduced representation $[f_0:\dots:f_n]$ of $f$.
Denote by $Q=\{1,\dots,q\}$. Assume that the decomposable hyperplanes $A_1, \dots, A_q$ are defined by  $A_1^*, \dots, A_q^*\in \Lambda^{k+1}(\mathbb{C}^{n+1})^{\vee}$, respectively.   
Let $\mathcal{S}$ be the set consisting of all subsets of $\{0,\dots,n\}$ having cardinality $k+1$, which in the lexicography order writes as
$
\mathcal{S}=
\{
I_0, I_1,\dots,I_{\mathfrak{n}-1}
\}
$. For every $I\in\mathcal{S}$, denote by $\|I\|$
its number of ranking,
so that $\|I_{i}\|=i+1$ for $0\leq i\leq\mathfrak{n}-1$.
For $I,J\in\mathcal{S}$, denote by $W(I,J)$ the determinant of the matrix $\big(f_j^{(i)}\big)_{i\,\in\, I,\, j\,\in\, J}$. Hence $W(I_0,J)$ coincides with the usual Wronskian $W(\{f_j\}_{j\in J})$. 
Let $\mathcal{W}=\big(W(I_r,I_s)\big)_{0\,\leq\, r,\,s\,\leq\, \mathfrak{n}-1}$ be the $(k+1)^{\text{th}}$-compound matrix of $\big(f_j^{(i)}\big)_{
	0\,\leq\, i,\,j\,\leq\, n}$.
Then the Sylvester--Franke theorem states that
\begin{align}
\label{Sylvester-Franke theorem application}
\det\mathcal{W}=W(f_0,\dots,f_n)^{\left(\substack{n\\k}\right)}.
\end{align}
Hence the zero order of $\det\mathcal{W}$ is well-defined, invariant under coordinate changes. In fact, its estimation will be a major challenge in this paper, and we will use some elaborate  coordinate system.

\subsection{An a priori estimate}
From now on, we assume that $q
-
2
\mathfrak{N}
+
\mathfrak{n}>0$, otherwise there is nothing to prove in the Main Theorem.

 Let $\{\omega(i)\}_{i\in Q}$  be the Nochka's weights and let $\widetilde{\omega}$ be the Nochka's constant of the family $\{A_i\}_{i\in Q}$. 
Recalling the construction of the derived curve $F_k$, we first
find some holomorphic function
$g$ whose zero divisor is
\begin{equation}
\label{D_k}
\mathcal{D}_k:=\min_{J\in\mathcal{S}}(W(I_0,J))_0.
\end{equation}

Here is an implement of Cartan's Wronskian technique and Nochka's estimate for derived curves.

\begin{pro}
	There exists some constant $C>0$ depending only on the family $\{A_i\}_{i\in Q}$ such that
	\begin{align}
\label{estimation of norm, statement}
\|F_k(z)\|_{\max}^{\widetilde{\omega}\big(q-2\,\mathfrak{N}+\mathfrak{n}\big)}
&\leq
C
\cdot
\bigg(
\dfrac{|g(z)|^{\mathfrak{n}}\prod_{i\in Q}| A_i^*\circ \overline{F}_k(z)|^{\omega(i)}}{|\det(\mathcal{W}(z))|}
\bigg)\notag\\
&\times\sum_{R\subset Q,\,\rank(R)=|R|=\mathfrak{n}}
\dfrac{|\det(\mathcal{W}(z))|}{\prod_{i\in R}|A_i^*\circ \widetilde{F}_k(z)|}.
	\end{align}
\end{pro}

\begin{proof} The arguments follow closely to that of \cite[page~125, Lem. 4.2.3]{Noguchi-Winkelmann2014}.
	Without loss of generality, we always assume that each hyperplane $A_i$ is defined by a linear form $A_i^*$ having unit norm $\|A_i^*\|=1$.
Since $\{A_i\}_{i\in Q}$ is in $(\mathfrak{N}-1)$-subgeneral position, for any point $[Z]\in \mathbb{P}^{\mathfrak{n-1}}(\mathbb{C})$, where $Z\in\mathbb{C}^{\mathfrak{n}}\setminus\{0\}$, there is some index subset $S\subset Q$ with cardinality $|S|=q-\mathfrak{N}$ such that all the corresponding hyperplanes miss $[Z]$, namely
$
\prod_{i\in S}\frac{A_i^*(Z)}{\|Z\|}\not=0
$. Noting that $A_i^*(Z)/\|Z\|$ is well-defined for $[Z]$, 
by compactness argument, there exists some constant $C_1>0$ depending only on $\{A_i\}_{i\in Q}$ such that
\begin{align}
\label{estimate norm1}
\dfrac{1}{C_1}
<
\sum_{S\subset Q,\,|S|=q-\mathfrak{N}}
\prod_{i\in S}\bigg(\dfrac{|A_i^*(Z)|}{\|Z\|}\bigg)^{\omega(i)}
<
C_1
\qquad\qquad\qquad
\scriptstyle{(\forall\,Z\,\in\,\mathbb{P}^{\mathfrak{n}-1}(\mathbb{C})).}
\end{align}
Denote by $C(S)$ the complement of $S$ in $Q$. Now we can rewrite each term in the middle of the above inequality as
\begin{equation}
\label{rewrite 1}
\prod_{i\in S}\bigg(\dfrac{|A_i^*(Z)|}{\|Z\|}\bigg)^{\omega(i)}
=
\dfrac{\prod_{i\in Q}|A_i^*(Z)|^{\omega(i)}}{\|Z\|^{\sum_{i\in Q}\omega(i)}}
\cdot
\prod_{i\in C(S)}\bigg(\dfrac{\|Z\|}{|A_i^*(Z)|}\bigg)^{\omega(i)}.
\end{equation}
Since $\|A_i^*\|=1$, we have $\frac{\|Z\|}{|A_i^*(Z)|}\geq 1$. Noting that $C(S)$ has cardinality $\mathfrak{N}$,
by the $(\mathfrak{N}-1)$-subgeneral assumption of $\{A_i\}_{i\in Q}$, we see that 
$\rank(C(S))=\mathfrak{n}$.
 Hence by   Proposition~\ref{Nochka estimate}, there exists an index subset $C_0(S)\subset C(S)$ having cardinality $\mathfrak{n}$ such that
\[
\prod_{i\in C(S)}\bigg(\dfrac{\|Z\|}{|A_i^*(Z)|}\bigg)^{\omega(i)}
\leq
\prod_{i\in C_0(S)}
\dfrac{\|Z\|}{|A_i^*(Z)|}.
\]
Remembering that
$\sum_{i\in Q}\omega(i)=\widetilde{\omega}(q-2\,\mathfrak{N}+\mathfrak{n})+\mathfrak{n}$ by Theorem~\ref{Nochka weights and constant},
we can estimate \eqref{rewrite 1} as
\begin{align}
\label{estimate norm 2}
\prod_{i\in S}\bigg(\dfrac{|A_i^*(Z)|}{\|Z\|}\bigg)^{\omega(i)}
&
\leq
\dfrac{\prod_{i\in Q}|A_i^*(Z)|^{\omega(i)}}{\|Z\|^{\widetilde{\omega}\big(q-2\,\mathfrak{N}+\mathfrak{n}\big)+\mathfrak{n}}}
\cdot
\prod_{i\in C_0(S)}
\dfrac{\|Z\|}{|A_i^*(Z)|}\notag\\
&=
\dfrac{\prod_{i\in Q}|A_i^*(Z)|^{\omega(i)}}{\|Z\|^{\widetilde{\omega}\big(q-2\,\mathfrak{N}+\mathfrak{n}\big)}}
\cdot
\dfrac{1}{\prod_{i\in C_0(S)}|A_i^*(Z)|}.\notag
\end{align}
Taking the sum on both sides of the above inequality for all $S$ and using the lower bound of~\eqref{estimate norm1}, we receive
\[
\|Z\|^{\widetilde{\omega}\big(q-2\,\mathfrak{N}+\mathfrak{n}\big)}
\leq
C_1
\cdot
\bigg(\prod_{i\in Q}|A_i^*(Z)|^{\omega(i)}\bigg)
\cdot
\sum_{S\subset Q,\,|S|=q-\mathfrak{N}}
\dfrac{1}{\prod_{i\in C_0(S)}|A_i^*(Z)|}.
\]
Substituting $Z$ by $\overline{F}_k(z)$ in the above inequality and noting that $\|F_k\|_{\max}\leq \|\overline{F}_k\|$, we receive
\begin{align}
\|F_k(z)\|_{\max}^{\widetilde{\omega}\big(q-2\,\mathfrak{N}+\mathfrak{n}\big)}
&\leq
C_1
\cdot
\bigg(\prod_{i\in Q}|A_i^*\circ \overline{F}_k(z)|^{\omega(i)}\bigg)
\cdot
\sum_{S\subset Q,\,|S|=q-\mathfrak{N}}
\dfrac{1}{\prod_{i\in C_0(S)}|A_i^*\circ \overline{F}_k(z)|}\notag\\
&=
C_1
\cdot
\bigg(
\dfrac{\prod_{i\in Q}|A_i^*\circ \overline{F}_k(z)|^{\omega(i)}}{|\det(\mathcal{W}(z))|}
\bigg)\times\sum_{S\subset Q,\,|S|=q-\mathfrak{N}}
\dfrac{|\det(\mathcal{W}(z))|\,
	|g(z)|^{\mathfrak{n}}}{\prod_{i\in C_0(S)}|A_i^*\circ \widetilde{F}_k(z)|}\notag\\
&\leq
C
\cdot
\bigg(
\dfrac{|g(z)|^{\mathfrak{n}}\prod_{i\in Q}|A_i^*\circ \overline{F}_k(z)|^{\omega(i)}}{|\det(\mathcal{W}(z))|}
\bigg)
\times\sum_{R\subset Q,\,|R|=\rank(R)=\mathfrak{n}}
\dfrac{|\det(\mathcal{W}(z))|}{\prod_{i\in R}|A_i^*\circ \widetilde{F}_k(z)|}\notag,
\end{align}
whence concludes the proof.
\end{proof}

\subsection{Fujimoto's vanishing order estimates}
In order to estimate the vanishing order of $\det \mathcal{W}$ effectively, Fujimoto~\cite[Section 5]{Fujimoto1982}
employed the following nice coordinate system.  The existence is essentially guaranteed by Gaussian elimination.

\begin{lem}
	\label{local expression of f}
	Let $f:\mathbb{C}\rightarrow\mathbb{P}^n(\mathbb{C})$ be a linearly nondegenerate entire holomorphic curve. For a given point $z_0\in\mathbb{C}$, there exist some homogeneous coordinates of $\mathbb{P}^n(\mathbb{C})$, a reduced representation of $f$, and a local coordinate $z$ in a small neighborhood $U$ of $z_0$ such that $f$ can be written as
	$f=[f_0:\dots:f_n]$, where 
	\begin{equation}
	\label{nice coordinates}
	f_i=z^{\alpha_i}+\sum_{j>\alpha_i}b_{ij}z^j
	\ \ \ \ \ \
	\scriptstyle{(b_{ij}\,\in\,\mathbb{C};\,0\,\leq\, i\,\leq\,n)}
	\end{equation}
	on $U$ and $\alpha_0=0<\alpha_1<\dots<\alpha_n$.
	\qed
\end{lem}

Thus, he received the following estimates, assuming a nice coordinate system for~\eqref{nice coordinates}.

\begin{cor}
	 One has
	$
	\mathcal{D}_k(z_0)=\sum_{i=0}^k(\alpha_i-i)
	$.
\end{cor}

\begin{defi}
	The {\sl weight} 
	$\mathsf{w}(I)$ of a set $I=\{i_0,\dots,i_k\}$ where $0\leq i_0<i_1<\dots<i_k<\infty$ is defined to be
	\[
	\mathsf{w}(I):=(i_0-0)+\dots+(i_k-k).
	\]
	\begin{rem}
		\label{estimate of the weight of I}
		If moreover $I
		\subset \{0,1,\dots, n\}$, then
		one has
		\[
		\mathsf{w}(I)
		\leq
		\mathsf{w}(\{n-k,n-k+1,\dots,n\})
		=
		(n-k)(k+1).
		\]
	\end{rem}
\end{defi}

\begin{cor}
	\label{zero order of W(I,J) estimate}
	For every $I,J\in S$, one has
	\[
	(W(I,J))_0
	\geq
	\sum_{a\in\mathbb{C}}
	\big(\mathcal{D}_k(a)-\mathsf{w}(I)+\mathsf{w}(J)
	\big)^+\{a\}.
	\]
\end{cor}

Running $I, J$ through $\mathcal{S}$, the summation of $-\mathsf{w}(I)$ and $\mathsf{w}(J)$ just cancel each other. Hence we obtain the following

\begin{pro}
	One has
	\begin{align}
		\label{rough estimate of divisors}
	(\det(\mathcal{W}))_0
	\geq
	\mathfrak{n}\,\mathcal{D}_k.
	\end{align}
\end{pro}

Since both sides of the above inequality is independent of  coordinates, it is in fact a general  
estimate.

\subsection{Fujimoto's trick}

Here is an essential ingredient in the proof of the key estimate~\eqref{estimate of logarithmic wronskian} below.

\begin{pro}[Fujimoto] \cite[Lem. 4.2]{Fujimoto1982}
	\label{expression of wronskian}
	Let $h_0,\dots,h_k$ be linearly independent meromorphic functions.
	Let $0\leq i_0<i_1<\dots<i_k$ be integers. Then the meromorphic function
	\[
	\frac{
	\det 
	(h^{(i_j)}_{\ell})_{j,\,\ell\,=\,0,\dots,\,k}
}{
	\det 
	(h^{(i)}_{\ell})_{i,\,\ell\,=\,0,\dots,\,k}
}
	\]
	can be written as a polynomial whose variables are of the form
	\[ 
	\label{variable of weight l}
	\bigg(
	\frac{
		\big(
		\det 
		(h_{\ell_i}^{(j)})_{i,\,j\,=\,0,\dots,\,r}
		\big)'
	}{\det\,
		(h_{\ell_i}^{(j)})_{i,\,j\,=\,0,\dots,\,r}
	}
	\bigg)^{(\lambda -1)}
	\eqno\scriptstyle{(0\,
		\leq\,r\,\leq\, k;\,\lambda\,\geq\, 1;\,0\,\leq\, \ell_0\,<\ell_1\,<\,\cdots\,<\,\ell_r\,\leq\, k).}
	\]
	Furthermore, if one associates weight $\lambda$ with the above variable, then this polynomial can be chosen to be isobaric of weight $\mathsf{w}(I)$ where $I=\{i_0,\dots, i_k\}$.
\end{pro}

\begin{cor}
	\label{typo cor}
	For any point $a\in \mathbb{C}$, one has
	\begin{equation}
	\label{new interpretation}
	\ord_a\,
		\det 
	(h^{(i_j)}_{\ell})_{j,\,\ell\,=\,0,\dots,\,k}
	\geq
	\ord_a\,
		\det 
	(h^{(i)}_{\ell})_{i,\,\ell\,=\,0,\dots,\,k}
	-
	\mathsf{w}(I).
	\end{equation}
\end{cor}

When $h_1,\dots, h_k$ are holomorphic, we can provide an alternative interpretation of~\eqref{new interpretation} as follows.  Considering the holomorphic map
\[
\mathsf{h}=[h_0:\cdots:h_k]
:
\mathbb{C}
\longrightarrow
\mathbb{P}^k(\mathbb{C}),
\]
by Lemma~\ref{local expression of f}, in a small neighborhood around $a\in\mathbb{C}$,
we can change the homogeneous coordinates of 
$\mathbb{P}^k(\mathbb{C})$  by some invertible matrix $P$ such that the representation
$
(f_0,\dots, f_k)=
P(h_0,\dots,h_0)
$
 of $\mathsf{h}$ in these coordinates satisfies
the increasing vanishing shape of~\eqref{nice coordinates}.
Thus by taking $n=k, J=I_0$
in Corollary~\ref{zero order of W(I,J) estimate}
 we receive
 \[
 \ord_a\,
 \det 
 (f^{(i_j)}_{\ell})_{j,\,\ell\,=\,0,\dots,\,k}
 \geq
 \ord_a\,
 \det 
 (f^{(i)}_{\ell})_{i,\,\ell\,=\,0,\dots,k\,}
 -
 \mathsf{w}(I),
 \]
 which in exactly~\eqref{new interpretation}.\qed

\subsection{A vanishing order estimate}

\begin{pro}
	\label{estimation of zero divisor of w}
	The following inequality  holds:
	\begin{align}
	\label{divisor inequality statement}
	\sum_{i\in Q}\omega(i)(A_i^*\circ \overline{F}_k)_0-(\det(\mathcal{W}))_0
	+\mathfrak{n}\,\mathcal{D}_k
	\leq
	\sum_{i\in Q}\omega(i)
\sum_{a\in\mathbb{C}}\min\{\ord_a A_i^*\circ \overline{F}_k,(k+1)(n-k)\}\{a\}
.
	\end{align}
\end{pro}

\begin{proof}
	The idea of the proof is to implement Nochka's weight technique~\cite{Nochka1983} in the course of Fujimoto's vanishing order estimates~\cite[Prop. 5.3]{Fujimoto1982}.
	Since
	\[
	\ord_a A_i^*\circ \overline{F}_k
	=
	\min\{\ord_a A_i^*\circ \overline{F}_k,(k+1)(n-k)\}
	+
	\big(\ord_a A_i^*\circ \overline{F}_k-(k+1)(n-k)\big)^+,
	\]
we can restate the inequality~\eqref{divisor inequality statement} as
\begin{align}
\label{estimate zero order of W, reduced}
(\det(\mathcal{W}))_0
\geq
\mathfrak{n}\mathcal{D}_k
+
\sum_{i\in Q}\omega(i)\sum_{a\in\mathbb{C}}\big(\ord_a A_i^*\circ \overline{F}_k-(k+1)(n-k)\big)^+\cdot\{a\}.
\end{align}	
It is a pointwise inequality, hence for every fixed $a\in\mathbb{C}$ we focus on the indices 
\[
S:=
\{
i\in Q: \ord_a A_i^*\circ \overline{F}_k
\geq
(k+1)(n-k)+1\},
\]
having nonzero contribution to the right hand side of~\eqref{estimate zero order of W, reduced}. 
By Corollary~\ref{rough estimate of divisors}, we only need to consider the case that $S\not=\emptyset$. Moreover, we claim that $|S|< \mathfrak{N}$. Indeed, suppose on the contrary that $S$ contains $\mathfrak{N}$ indices, say $1, \dots, \mathfrak{N}$. By the assumption of subgeneral position, the corresponding hyperplanes $A_1, \dots, A_{\mathfrak{N}}$ have empty intersection, hence
at least one $A_i^*\circ \overline{F}_k(a)\neq 0$, contradicting to the selection of $S$. 
 
Now we exhibit the distinct values $\{\ord_a A_i^*\circ \overline{F}_k\}_{i\in S}$ from high to low
\[
m_1>m_2>\dots>\dots>m_t
,\]
 and then set a filtration of  $S$ accordingly
$
 S_0:=\emptyset\subset S_1\subset\dots\subset S_{t}=S
$,
where  for every $i\in S_{\ell}\setminus S_{\ell-1}$, there holds $\ord_aA_i^*\circ \overline{F}_k = m_{\ell}$, respectively for $\ell=1,\dots, t$. Let $\{T_{\ell}\subset S_{\ell}\}_{\ell\,=\,1,\dots,\, t}$ be a family of increasing subsets $T_1\subset \cdots\subset T_t$ constructed subsequently  by the law $|T_{\ell}|=\rank(T_{\ell})=\rank(S_{\ell})$.
 Set $ \widetilde{m_{\ell}}=m_{\ell}-(k+1)(n-k)$. 
Now we can estimate
\begin{align}
\label{first estimate, remove nochka weights}
&\ \ \ \ \ \sum_{i\in Q}\omega(i)
\big(
\ord_a A_i^*\circ \overline{F}_k-(k+1)(n-k)
\big)^+\notag\\
&=
\sum_{i\in S}\omega(i)
\big(
\ord_a A_i^*\circ \overline{F}_k-(k+1)(n-k)
\big)\notag\\
&=
\sum_{\ell=1}^{t}\sum_{i\in S_{\ell}\setminus S_{\ell-1}}
\omega(i)\,
\widetilde{m_{\ell}}\notag\\
&=
\sum_{\ell=1}^{t}
\Big{(}
\sum_{i\in S_{\ell}}
\omega(i)\,
\widetilde{m_{\ell}}
-
\sum_{j\in S_{\ell-1}}
\omega(j)\,
\widetilde{m_{\ell}}
\Big{)}
\notag\\
&=
(\widetilde{m_1}-\widetilde{m_2})\sum_{i\in S_1}\omega(i)
+
(\widetilde{m_2}-\widetilde{m_3})\sum_{i\in S_2}\omega(i)
+
\dots
+
\widetilde{m_t}\sum_{i\in S_{t}}\omega(i)\notag\\
\text{[by~\eqref{Nochka rank inequatlity}]}\ \ \ \ \ 
&\leq
(\widetilde{m_1}-\widetilde{m_2})\rank(S_1)
+
(\widetilde{m_2}-\widetilde{m_3})\rank(S_2)
+
\dots
+
\widetilde{m_t}\rank(S_{t})\notag\\
&=
\rank(S_1)\widetilde{m_1}+
\big(\rank(S_2)-\rank(S_1)\big)\widetilde{m_2}
+
\dots
+
\big(\rank(S_{t})-\rank(S_{t-1})\big)\widetilde{m_t}\notag\\
&=
|T_1|\widetilde{m_1}
+
(|T_2|
-
 |T_1|)\widetilde{m_2}
+
\dots
+(|T_{t}|- T_{t-1}|)\widetilde{m_t}
\notag
\\
&=
|T_1|\widetilde{m_1}
+
|T_2\setminus T_1|\widetilde{m_2}
+
\dots
+|T_{t}\setminus T_{t-1}|\widetilde{m_t}.
\end{align}

Changing the indices of hyperplanes $\{A_i\}_{i\in Q}$ if necessary, we may assume that $T_{t}=\{1,\dots,|T_{t}|\}$. Set  $m^*_s=\ord_aA_s^*\circ \overline{F}_k$ for $s=1,\dots, |T_{t}|$. Then~\eqref{first estimate, remove nochka weights} reads as
\begin{align}
\label{remove nochka estimate, change order}
\sum_{i\in Q}\omega(i)
\big(
\ord_a A_i^*\circ \overline{F}_k-(k+1)(n-k)
\big)^+
\leq
\sum_{s=1}^{|T_{t}|}
\big(
m_s^*-(k+1)(n-k)
\big).
\end{align}
Hence the desired inequality~\eqref{estimate zero order of W, reduced} can be established by showing a stronger estimate
\begin{equation}
\label{desired estimate}
\ord_a\det(\mathcal{W})
\geq
\mathfrak{n}\,\mathcal{D}_k(a)
+
\sum_{s=1}^{|T_{t}|}
\big(
m_s^*-(k+1)(n-k)
\big).
\end{equation}

We first recall that a similar a priori estimate~\eqref{rough estimate of divisors}
can be achieved by applying Lemma~\ref{local expression of f}.
Indeed, we can calculate the vanishing orders of $\det(\mathcal{W})$ and $\mathcal{D}_k$ at the given point $a$ more effectively by  means of nice coordinates, in which $f_0,\dots,f_n$ have explicit increasing vanishing shapes as~\eqref{nice coordinates}. Thus for every $I, J\in \mathcal{S}$ the $(\|I\|,\|J\|)$-th entry of $\mathcal{W}$ has vanishing order $\geq \mathcal{D}_k-\mathsf{w}(I)+\mathsf{w}(J)$, whence $\det \mathcal{W}$ satisfies the
 estimate~\eqref{rough estimate of divisors} by straightforward summation based on the Laplace expansion. 
But to reach the stronger estimate~\eqref{desired estimate} we need more effort, inevitably by exploiting the extra condition that $\{A_j^*\circ \overline{F}_k\}_{j=1,\dots,\, |T_t|}$
have high vanishing orders. 
 Here is our strategy. 
 We will modify some $|T_t|$ columns of $\mathcal{W}$
to represent the information of $\{A_j^*\circ \overline{F}_k\}_{j\,=\,1,\dots,\, |T_t|}$,
by multiplying certain well-chosen invertible matrix $\mathsf{I}$.
Thus the new obtained matrix $\widetilde{\mathcal{W}}=\mathcal{W}\cdot\mathsf{I}$ keeps the same vanishing order of determinant. Now for $s=1,\dots, |T_t|$ the $s$-th ``new column" of 
$\widetilde{\mathcal{W}}$ contribute, in each entry, at least $m_s^*-(k+1)(n-k)$ more vanishing order estimate than that of $\mathcal{W}$. Whence by counting vanishing order in each term of the Laplacian expansion of $\det (\widetilde{\mathcal{W}})_0$, we conclude the proof.

Now we carry out the details. Starting with the following 

\smallskip\noindent
{\bf Fact.}
{\it
	Let $\{v_i\}_{i=1,\dots,\, m }$ be a basis of a linear space $V$, and let
	$\widetilde{v}_1,\dots,\widetilde{v}_{\ell}\in V$ be some linearly independent vectors. Then one can replace some $\ell$ vectors in 
	$\{v_i\}_{i=1,\dots,\, m }$ by $\widetilde{v}_1,\dots,\widetilde{v}_{\ell}$
	such that they still form a basis.
	\qed
}

Applying the above fact to $V=\Lambda^{k+1}(\mathbb{C}^{n+1})^{\vee}$ and its basis $\{e^*_{I_i}=\wedge_{\ell\in I_i} e^*_{\ell}\}_{i\,=\,0,\,1,\dots,\, \mathfrak{n}-1}$,
we can replace some $|T_t|$ vectors $e_{I_{i_{1}}}^*,\dots, 
e_{I_{i_{|T_t|}}}^*$ by 
 $A_1^*,\dots, A_{|T_t|}^*$ respectively to
 receive a new basis 
 \begin{equation}
 \label{change basis}
 (b_1,\dots, b_{\mathfrak{n}})=(e^*_{I_0},\dots, e^*_{I_{\mathfrak{n}-1}})
 \cdot
 \mathsf{I},
 \end{equation}
where according to our construction, $\mathsf{I}$
differs from the identity matrix 
only in the columns $i_1+1,\dots, i_{|T_t|}+1$.

Write 
 $A_1^*=l_0\wedge \cdots \wedge l_k$,
 where linear forms $l_j\in (\mathbb{C}^{n+1})^{\vee}$
comparing to the standard basis $\{e_j^*\}_{j=0,\dots,\, n}$ read as $(l_0,\dots, l_k)=(e_0^*,\dots, e_n^*)\cdot L$ for  some $(n+1)\times (k+1)$ matrix $L$.
By~\eqref{wedge product formula} we have
\begin{equation}
\label{A1*=}
A_1^*
=
\sum_{0\,\leq\, j_0\,<\cdots <\,j_k\,\leq\, n} \det(L_{\{j_0,\dots,\,j_k\}})
\,
e_{j_0}^*\wedge\cdots\wedge e_{j_k}^*
=
\sum_{i=0}^{\mathfrak{n}-1}\det(L_{I_i})
\,
e^*_{I_i},
\end{equation}
where $L_{\{j_0,\dots,\,j_k\}}$ consists of the rows $j_0+1,\dots, j_k+1$ of $L$.
This shows all the entries of the $(i_1+1)$-th column of $\mathsf{I}$, and hence  the $(\|J\|, i_1+1)$-th entry of $\widetilde{\mathcal{W}}=\mathcal{W}\cdot\mathsf{I}$ is nothing but:
\begin{align}
\sum_{i=0}^{\mathfrak{n}-1}\det(L_{I_i})
\,
W(J,I_i)
&
=
\sum_{I\in \mathcal{S}}\det(L_{I})
\,
\det
\big( (f^{(j)}_{\ell})_{j\in J,\, \ell\in I}
\big)
\label{A1Fk}
\\
\text{[by Cauchy-Binet Formula]\ \ \ }
&
=
\det
\big( (h^{(j)}_{\ell})_{j\,\in\, J,\, \ell\,=\,0,\dots,\, k}
\big)
\label{det(AF)},
\end{align}
where similar to~\eqref{use Cauchy-Binet formula} we have
\begin{equation}
\label{use Cauchy-Binet formula again}
(h_0,\dots,h_k)
=
(l_0\circ\widetilde{f},\dots, l_k\circ\widetilde{f})
\end{equation}
for the lifting $\widetilde{f}$ given in~\eqref{lifting f}.
Noting that $h^{(j)}_{\ell}=l_{\ell}\circ \widetilde{f}^{(j)}$,
setting $\widetilde{F}^{J}:=\wedge_{j\in J}\widetilde{f}^{(j)}$,
then~\eqref{det(AF)} becomes
\begin{align}
\det
\big( (h^{(j)}_{\ell})_{j\,\in\, J,\, \ell\,=\,0,\dots,\, k}
\big)
&=
\det
\big( (l_{\ell}\circ \widetilde{f}^{(j)})_{j\,\in\, J,\, \ell\,=\,0,\dots,\, k}
\big)
\notag
\\
&
=
(l_0\wedge \cdots \wedge l_k)\cdot (\wedge_{j\in J}\widetilde{f}^{(j)})
\notag
\\
&
=A_1^*\circ \widetilde{F}^{J}
\label{neat A_1*F^J}.
\end{align}

In particular, for $J=I_0=\{0,\dots,k\}$, 
the $(1, i_1+1)$-th entry of $\widetilde{\mathcal{W}}$ is 
$ A_1^*\circ \widetilde{F}_k=g\cdot( A_1^*\circ \overline{F}_k)$,
which is known to have high vanishing order $\mathcal{D}_k(a)+m_1^*$ at the point $a$. Lastly, applying Corollary~\ref{typo cor} upon the neat determinant~\eqref{det(AF)} and using Remark~\ref{estimate of the weight of I},
 we conclude that for any $J\in \mathcal{S}$ the $(\|J\|, i_1+1)$-th entry of $\widetilde{\mathcal{W}}$ has vanishing order at least 
 \begin{equation}
 \label{use typo cor}
 \mathcal{D}_k(a)+m_1^*-\mathsf{w}(J)
 \geq
 \Big( \mathcal{D}_k(a)-\mathsf{w}(J)+
 \mathsf{w}(I_{i_1})
 \Big)
 +
 \Big(
 m_1^*-(k+1)(n-k)
 \Big).
 \end{equation}
 Similarly, for $s=1,\dots,|T_t|$, the same argument shows that the $(\|J\|, i_s+1)$-th entry of $\widetilde{\mathcal{W}}$ has vanishing order at least :
  \[
  \mathcal{D}_k(a)+m_s^*-\mathsf{w}(J)
  \geq
  \Big( \mathcal{D}_k(a)-\mathsf{w}(J)+
  \mathsf{w}(I_{i_s})
  \Big)
  +
  \Big(
  m_s^*-(k+1)(n-k)
  \Big).
  \]
   Note that the first bracket above is exactly  the original vanishing order estimate of the $(\|J\|, i_s+1)$-th entry of $\mathcal{W}$, and that the second bracket is a summand in~\eqref{desired estimate}.
 By straightforward summation based on the Laplace expansion, we conclude the proof.
\end{proof} 

\subsection{An application of the logarithmic derivative lemma}
Here is an estimate due to Fujimoto~\cite{Fujimoto1982}.

\begin{pro}
	\label{psi estimate}	
	One has the estimate
	\begin{equation}
	\label{estimate of logarithmic wronskian}
	\dfrac{1}{2\pi}\int_0^{2\pi}
	\max_{R\subset Q,\,|R|=\rank(R)=\mathfrak{n}}
	\log
	\dfrac{|\det(\mathcal{W})|}{\prod_{i\in R}| A_i^*\circ \widetilde{F}_k|}
	(re^{i\theta})
	\dif\theta
	=S_{F_k}(r).
	\end{equation}
\end{pro}

For the sake of completeness, we include a proof here. To start with, let us recall

\begin{namedthm*}{Logarithmic derivative Lemma}
	\cite[Lem. 4.2.9]{Noguchi-Winkelmann2014}
	Let $g$ be a nonconstant meromorphic function on $\mathbb{C}$. Then for any integer $\ell\geq 1$, the following estimate holds
	\[
	m_{\big(\frac{g'}{g}\big)^{(\ell)}}(r)
	:=
	m_{\big(\frac{g'}{g}\big)^{(\ell)}}(r)(r,\infty)=S_g(r).
	\]	
\end{namedthm*}

To prove~\eqref{estimate of logarithmic wronskian}, one must get rid of $g$ in the left-hand side. Hence it is necessary to work in logarithmic setting. Taking the wedge products of the logarithmic derivatives
\[
{f}_{\log}^{(\ell)}
=
\bigg(
\Big(\dfrac{f_0}{f_0}\Big)^{(\ell)},\dots,\Big(\dfrac{f_n}{f_0}\Big)^{\ell)}
\bigg)
:\ \ \mathbb{C} \ \ \longrightarrow
\ \ \mathbb{C}^{n+1}
\eqno
\scriptstyle{(\ell\,=\,0,\,1,\dots,\, k)},
\]
we obtain the logarithmic derived curve
\[
\widetilde{F}_{k,\log}
=
{f}_{\log}^{(0)}
\wedge\cdots\wedge
{f}_{\log}^{(k)}
:\ \ \mathbb{C} \ \ \longrightarrow
\ \ \Lambda^{k+1}\,\mathbb{C}^{n+1},
\]
which in Pl\"{u}cker coordinates reads as
\[
\widetilde{F}_{k,\log}
=
\sum_{0\,\leq\, i_0\,<\,i_1\,<\,\cdots\,<\,i_k\,\leq\, n} W_{\log}(f_{i_0},\dots,f_{i_k})\,e_{i_0}\wedge\dots\wedge e_{i_k},
\]
where 
\[
W_{\log}(f_{i_0},\dots,f_{i_k})
:=
\det
\Big(
(f_{i_{\beta}}/f_0)^{(\alpha)}
\Big)_{\alpha,\,\beta\,=\,0,\dots,\, k }
=f_0^{-(k+1)}W(f_{i_0},\dots,f_{i_k})
\]
is the logarithmic Wronskian. Hence we have
\begin{equation}
\label{bar Fk and widetilde Fk compare}
\widetilde{F}_{k,\,\log}
=
f_0^{-(k+1)}\widetilde{F}_k.
\end{equation}

For $I,J\in\mathcal{S}$, the logarithmic analog $W_{\log}(I,J)$ of $W_{I,J}$ is defined to be the determinant of the matrix $\big(({f_j}/{f_0})^{(i)}\big)_{i\in I, j\in J}$.
Setting
\[
\mathcal{W}_{\log}:=
\big(W_{\log}(I_r,I_s)\big)_{0\,\leq\, r,\,s\,\leq\, \mathfrak{n}-1},
\]
by Sylvester-Franke theorem
we have
\begin{align}
\det(\mathcal{W}_{\log})
&=
W_{\log}(f_0,\dots,f_n)^{\left(\substack{n\\k}\right)}
\notag
\\
&=
\Big(
f_0^{-(n+1)}
W(f_0,\dots,f_n)
\Big)^{\left(\substack{n\\k}\right)}
\notag
\\
\label{det W and det Wlog compare}
\text{[recall~\eqref{Sylvester-Franke theorem application}]}\ \ \ \ \ \
&=
f_0^{-(k+1)\mathfrak{n}}
\det(\mathcal{W}),
\end{align}
where in the last equality
we need a straightforward calculation
$(n+1)\left(\substack{n\\k}\right)=(k+1)\left(\substack{n+1\\k+1}\right)=(k+1)\mathfrak{n}$.

\textit{Proof of Proposition~\ref{psi estimate}.} 
	By~\eqref{bar Fk and widetilde Fk compare}, \eqref{det W and det Wlog compare}, we rewrite
	\[
	\dfrac{|\det(\mathcal{W})|}{\prod_{i\in R}| A_i^*\circ \widetilde{F}_k|}
	=
	\dfrac{|\det(\mathcal{W}_{\log})|}{\prod_{i\in R}| A_i^*\circ \widetilde{F}_{k,\, \log}|}.
	\]
Hence
\begin{equation}
\label{maxlog det w/Ai FK estimate}
\max_{R\subset Q,\,|R|=\rank(R)=\mathfrak{n}}
\log
\dfrac{|\det(\mathcal{W})|}{\prod_{i\in R}|A_i^*\circ \widetilde{F}_k|}
\leq
\sum_{R\subset Q,\,|R|=\rank(R)=\mathfrak{n}}
\log^+
\dfrac{|\det(\mathcal{W}_{\log})|}{\prod_{i\in R}|A_i^*\circ \widetilde{F}_{k,\,\log}|}.
\end{equation}	

We now analyze each summand above.
Without loss of generality, we illustrate by $R=\{1, \dots,\mathfrak{n}\}$ for simplicity of indices.
Since $\{A_i^*\}_{i\in R}$ form a basis for
$\Lambda^{k+1}(\mathbb{C}^{n+1})^{\vee}$, changing coordinates we read
\[
(A_1^*,\dots, A_{\mathfrak{n}}^*)
=
(e^*_{I_0},\dots, e^*_{I_{\mathfrak{n}-1}})
\cdot
\mathsf{C}
\]
for an invertible $\mathfrak{n}\times \mathfrak{n}$-matrix $\mathsf{C}$.
Now the matrix
$\widetilde{\mathcal{W}}_{\log}:=\mathcal{W}_{\log}\cdot \mathsf{C}$, similar to 
 $\widetilde{\mathcal{W}}$ below~\eqref{change basis}, 
 has a neat expression of determinant in each entry. Indeed,
 setting $\widetilde{F}^I_{ \log}:=\wedge_{i\in I}{f}^{(i)}_{\log}$ for every $I \in\mathcal{S}$, by the same argument as~\eqref{A1Fk}, \eqref{det(AF)},
 the $(|I|, j)$-th entry of $\widetilde{\mathcal{W}}_{\log}$ is nothing but
 $A_j^*\circ\widetilde{F}^I_{ \log}$.
Thus
\begin{align}
\dfrac{|\det(\mathcal{W}_{\log})|}{\prod_{i\in R}|A_i^*\circ \widetilde{F}_{k,\,\log}|}
&=
\dfrac{|\det(\widetilde{\mathcal{W}}_{\log}\cdot\mathsf{C}^{-1})|}{\prod_{i\,=\,1,\dots,\, \mathfrak{n}}|A_i^*\circ \widetilde{F}_{k,\,\log}|}
\notag
\\
&=
|\det(\mathsf{C}^{-1})|
\times
\dfrac{
	\Big|
	\det
	\big(
	A_{j}^*\circ\widetilde{F}^{I_{i-1}}_{\log}
	\big)_{i,\,j\,=\,1,\dots,\,\mathfrak{n}}
\Big|
}
{
\prod_{j\,=\,1,\dots,\, \mathfrak{n}}|A_j^*\circ \widetilde{F}_{k,\,\log}|}
\notag\\
&=
|\det(\mathsf{C}^{-1})|
\times
\bigg|
\det
	\bigg(
	\dfrac{A_{j}^*\circ\widetilde{F}^{I_{i-1}}_{\log}}{A_j^*\circ \widetilde{F}_{k,\,\log}}
	\bigg)_{i,\,j\,=\,1,\dots,\,\mathfrak{n}}
\bigg|.
\label{logarithimic trick}
	\end{align}
Using the basic inequalities
\begin{equation}
\label{basic inequality log+}
\log^+\bigg(\sum_{i=1}^px_i\bigg)
\leq\sum_{i=1}^p\log^+x_i+\log p,\qquad\qquad
\log^+\bigg(\prod_{i=1}^px_i\bigg)
\leq 
\sum_{i=1}^p\log^+x_i,
\end{equation}
we have
\begin{align}
\label{logarithmic wronskian term by term}
\log^+
\dfrac{|\det(\mathcal{W}_{\log})|}{\prod_{i\in R}|A_i^*\circ \widetilde{F}_{k,\,\log}|}
\leq
\sum_{i,\,j\,=\,1,\dots,\,\mathfrak{n}}
\log^+
\bigg|
\dfrac{A_{j}^*\circ\widetilde{F}^{I_{i-1}}_{\log}}{A_j^*\circ \widetilde{F}_{k,\,\log}}
\bigg|
+C,
\end{align}	
where $C$ is some constant independent of $f$.	Now the problem reduces to showing that
\begin{equation}
\label{reduced problem, logarithmic wronskian Ai FIr/Ai Fk}
\dfrac{1}{2\pi}\int_0^{2\pi}
\log^+
\bigg|
\dfrac{A_{j}^*\circ\widetilde{F}^{I}_{\log}}{A_j^*\circ \widetilde{F}_{k,\,\log}}
\bigg|
(re^{i\theta})
\dif\theta
=
S_{F_k}(r),
\end{equation}
for any $I\in  \mathcal{S}$ and $j\in Q$.
We illustrate by $j=1$ and $A_1^*=a_0\wedge\cdots \wedge a_k$, where
each $a_i\in (\mathbb{C}^{n+1})^{\vee}$ in the standard dual basis $\{e_j^*\}_{j\,=\,0,\dots,\, n}$  
reads as
$a_i=\sum_{j=0}^n a_{i,j}e_j^*$.
Similar to~\eqref{det(AF)}, we have
\[
A_{j}^*\circ\widetilde{F}^{I}_{\log}
=
\det
\big( (h^{(i)}_{\ell})_{i\,\in\, I,\, \ell\,=\,0,\dots,\, k}
\big)
\]
where each $h_{\ell}=\sum_{i=0}^n
a_{\ell,i}\, f_i/f_0$. Now by applying Proposition~\ref{expression of wronskian} and the Logarithmic Derivative Lemma, the desired estimate~\eqref{reduced problem, logarithmic wronskian Ai FIr/Ai Fk} follows directly from~\eqref{same growth rate for derived curves}, ~\eqref{basic inequality log+} and the following

\smallskip\noindent
{\bf Fact}. \cite[page.~78, Thm. 2.5.13]{Noguchi-Winkelmann2014}
For every $i=0,1,\dots ,n$, one has the estimate
\[
T(r, f_{i}/f_0)\leq O(T_f(r)).
\]
\qed

\subsection{End of the proof of the Main Theorem}

Taking logarithm on both sides of~\eqref{estimation of norm, statement} and then integrating, we receive
\begin{align}
\label{T<=integer varphi+psi}
\widetilde{\omega}\big(q-2\,\mathfrak{N}+\mathfrak{n}\big)T_{F_k}(r)
\leq
\dfrac{1}{2\pi}
\int_0^{2\pi}
\log \varphi(re^{i\theta})
\dif\theta
+
\dfrac{1}{2\pi}
\int_0^{2\pi}
\psi(re^{i\theta})\dif\theta
+O(1),
\end{align}
where
\[
\varphi
=
\dfrac{|g|^{\mathfrak{n}}\prod_{i\in Q}| A_i^*\circ \overline{F}_k|^{\omega(i)}}{|\det(\mathcal{W})|},
\ \ \ \ \ 
\psi=
\sum_{R\subset Q,\,\rank(R)=|R|=\mathfrak{n}}
\dfrac{|\det(\mathcal{W})|}{\prod_{i\in R}|A_i^*\circ \widetilde{F}_k|}.
\]
Using Proposition~\ref{estimation of zero divisor of w}, we receive
\begin{align*}
(\varphi)_0
&=
\sum_{i\in Q}\omega_i(A_i^*\circ \overline{F}_k)_0
+
\mathfrak{n}(\mathcal{D}_k)_0-(\det\mathcal{W})_0\\
&\leq
\sum_{i\in Q}\omega(i)
\sum_{a\in\mathbb{C}}\min\{\ord_a A_i^*\circ \overline{F}_k,(k+1)(n-k)\}\{a\}\\
&\leq
\widetilde{\omega}\sum_{i\in Q}
\sum_{a\in\mathbb{C}}\min\{\ord_a A_i^*\circ \overline{F}_k,(k+1)(n-k)\}\{a\}.
\end{align*}
Whence by Jensen formula we have
\begin{align}
\dfrac{1}{2\pi}
\int_0^{2\pi}
\log|\varphi(re^{i\theta})|\dif\theta
&\leq
N_{\varphi}(r,0)+O(1)\notag\\
&\leq
\widetilde{\omega}
\sum_{i\in Q} N_{F_k}^{[(k+1)(n-k)]}(r,A_i)
+
O(1).\notag
\end{align}

Together this with~\eqref{T<=integer varphi+psi} and Proposition~\ref{psi estimate},
we finish the proof.

\section{Some applications}
\label{section: applications}
\subsection{A defect relation}

\begin{namedthm*}{Defect relation}
	Let $f:\mathbb{C}\rightarrow\mathbb{P}^n(\mathbb{C})$ be a linearly nondegenerate  entire holomorphic curve. For a fixed integer $k=0,1,\dots, n$, let $A_1,\dots, A_q\subset \mathbb{P}
	\big(
	\Lambda^{k+1}(\mathbb{C}^{n+1})
	\big)$ be  $q$ decomposable hyperplanes
	such that any $\mathfrak{N}$ of them have empty intersection. 
	  Then the $k$-th derived curve $F_k$ of $f$ satisfy the following estimate:
		\begin{equation}
	\label{defect relation}
	\sum_{i=1}^q \delta_{F_k}^{[(k+1)(n-k)]}(A_i)
	\leq
2
\mathfrak{N}
-
\mathfrak{n}.
\end{equation}
	\end{namedthm*}

\begin{proof}
	The Main Theorem can be rewritten as
	\[
		\sum_{i=1}^q
	\bigg(
	1-
	\dfrac{N_{F_k}^{[(k+1)(n-k)]}(r,A_i)}{T_{F_k}(r)}
	\bigg)
\leq
	2
	\mathfrak{N}
	-
	\mathfrak{n}
	+
	\dfrac{S_{F_k}(r)}{T_{F_k}(r)}.
	\]
	Taking the limit inferior of both sides of the above inequality, we conclude the proof.
\end{proof}

\subsection{Ramification Theorem}

\begin{thm}
	 In the setting of the Main Theorem, assuming moreover that
	 the associated $k$-th derived curve $F_k$ is completely $\mu_{k,i}$--ramified over each decomposable hyperplane $A_i$ for $i=1,\dots, q$, then one has
	\[
	\label{ramification theorem statement}
	\sum_{i=1}^q
	\bigg(
	1
	-
	\dfrac{(k+1)(n-k)}{\mu_{k,i}}
	\bigg) 
	\leq
2
\mathfrak{N}
-
\mathfrak{n}.
	\]
\end{thm}

\begin{proof}
	For an index $i$ with $\mu_{k,i}<\infty$,  every nonzero coefficients of the divisor $(A_i^*\circ \overline{F}_k)_0$ is $\geq \mu_i^k$. Hence
	\begin{align*}
	\delta_{F_k}^{[(k+1)(n-k)]}(A_i)
	&=
	1-
	\limsup\dfrac{N_{F_k}^{[(k+1)(n-k)]}(r,A_i)}{T_{F_k}(r)}
	\\
	&\geq
	1-(k+1)(n-k)
	\limsup
	\dfrac{N_{F_k}^{[1]}(r,A_i)}{T_{F_k}(r)}\\
	\explain{By the First Main Theorem}
	\ \ \ \ \ \ \ \ \ \ \ \ \
	&\geq
	1-(k+1)(n-k)
	\limsup
	\dfrac{N_{F_k}^{[1]}(r,A_i)}{N_{F_k}(r,A_i)}\\
	&\geq
	1-\dfrac{(k+1)(n-k)}{\mu_i^k}.
	\end{align*}
	When $\mu_i^k=\infty$, the above inequality is trivial. By the defect relation we finish the proof.
\end{proof}

\begin{center}
\bibliographystyle{plain}
\bibliography{reference}
\end{center}
\Addresses
\end{document}